\documentclass[oneside,russian,english,reqno]{amsart}
\usepackage{mathpazo}
\usepackage[T1]{fontenc}
\usepackage[koi8-r,latin9]{inputenc}
\usepackage{color}
\usepackage{babel}
\usepackage{amsthm}
\usepackage{amstext}
\usepackage{amssymb}
\usepackage[unicode=true,pdfusetitle,
 bookmarks=true,bookmarksnumbered=false,bookmarksopen=false,
 breaklinks=false,pdfborder={0 0 0},backref=false,colorlinks=true]
 {hyperref}
\hypersetup{
 pdfborderstyle={},linkcolor=black,citecolor=black,urlcolor=blue}
\usepackage{breakurl}

\makeatletter

\DeclareRobustCommand{\cyrtext}{%
  \fontencoding{T2A}\selectfont\def\encodingdefault{T2A}}
\DeclareRobustCommand{\textcyr}[1]{\leavevmode{\cyrtext #1}}
\AtBeginDocument{\DeclareFontEncoding{T2A}{}{}}

  \theoremstyle{plain}
  \newtheorem*{thm*}{\protect\theoremname}
\theoremstyle{plain}
\newtheorem{thm}{\protect\theoremname}
 \theoremstyle{definition}
 \newtheorem*{defn*}{\protect\definitionname}

\usepackage{calrsfs}
\usepackage{graphicx}
\usepackage{color}
\usepackage{perpage}
\usepackage{upquote}

\MakePerPage{footnote}
\gdef\SetFigFontNFSS#1#2#3#4#5{} 

\DeclareMathOperator{\supp}{supp}

\theoremstyle{remark}
\newtheorem*{qst*}{Question}
\newtheorem*{rmrks*}{Remarks}

\makeatother

  \addto\captionsenglish{\renewcommand{\definitionname}{Definition}}
  \addto\captionsenglish{\renewcommand{\theoremname}{Theorem}}
  \addto\captionsrussian{\renewcommand{\definitionname}{\inputencoding{koi8-r}ïÐÒÅÄÅÌÅÎÉÅ}}
  \addto\captionsrussian{\renewcommand{\theoremname}{\inputencoding{koi8-r}ôÅÏÒÅÍÁ}}
  \providecommand{\definitionname}{Definition}
  \providecommand{\theoremname}{Theorem}
\providecommand{\theoremname}{Theorem}

\begin{document}

\title{Singular distributions and symmetry of the spectrum}

\author{Gady Kozma and Alexander Olevskii}

\address{GK: Department of Math, The Weizmann Institute of Science, Rehovot
76100, Israel.}

\address{AO: School of Mathematics, Tel Aviv University, Tel Aviv 69978, Israel.}

\email{gady.kozma@weizmann.ac.il, olevskii@post.tau.ac.il}

\maketitle
This is a survey of the ``Fourier symmetry'' of measures and distributions
on the circle in relation with the size of their support.  Mostly
it is based on our paper \cite{KO} and a talk given by the second
author in the 2012 Abel symposium.

\section{Introduction}

Below we denote by $S$ a Schwartz distribution on the circle group
$\mathbb{T}$; by $K_{S}$ its support; and by $\widehat{S}(n)$ its
Fourier transform. It has polynomial growth. $S$ is called a \emph{pseudo-function}
if $\widehat{S}(n)=o(1)$. In this case the Fourier series
\[
\sum_{n=-\infty}^{\infty}\widehat{S}(n)e^{int}
\]
converges to zero pointwisely on $^{c}K$, see \cite{KS94}. 

The purpose of this survey is to highlight the phenomenon that some
fundamental properties of a trigonometric series depend crucially
on the ``balance'' between its analytic and anti-analytic parts
(which correspond to the positive and negative parts of the spectrum).
We start with a classic examples, comparing Menshov's and Privalov's
theorems.
\begin{thm*}
[Menshov, \cite{M16}]There is a (non-trivial) singular, compactly
supported measure $\mu$ on the circle which is a pseudo-function. 
\end{thm*}
See \cite[\S XIV.12]{B64}. In particular, it means that a non-trivial
trigonometric series 
\begin{equation}
\sum_{n\in Z}c(n)e^{int}\label{eq:1}
\end{equation}
may converge to zero almost everywhere. This disproved a common belief
that the uniqueness results of Riemann and Cantor may be strengthened
to any trigonometrique series converging to zero almost everywhere
(shared by Lebesgue, \cite{L1902}). Contrast Menshov's theorem with
the following result:
\begin{thm*}
[Abel \& Privalov]An ``analytic'' series 
\begin{equation}
\sum_{n\geq0}c(n)e^{int}\label{eq:2}
\end{equation}
can not converge to zero on a set of positive measure, unless it is
trivial. 
\end{thm*}
The theorem of Abel \cite[\S 3.14]{Z68} gives that if $\sum c(n)e^{int}$
converges at some $t$, then the function $\sum c(n)z^{n}$, which
is analytic in the disc $\{|z|<1\}$, converges \emph{non-tangentially}
to the same value at $e^{it}$. The theorem of Privalov \cite[\S D.III]{K80}
claims that an analytic function on the disc which converges non-tangentially
to zero on a set of positive measure is identically zero. Together
these two results give the theorem above. 

We are interested in different aspects of symmetry and non-symmetry
of the Fourier transform of measures and distributions. Our first
example is Frostman's theorem.

\section{Frostman's theorem}

The classic Frostman theorem connects the Hausdorff dimension of a
compact set $K$ to the (symmetric) behaviour of the Fourier coefficients
of measures supported on $K$. Let us state it in the form relevant
to us.
\begin{thm*}
[Frostman]\hfill\label{thm:1}
\begin{enumerate}
\item If a compact $K$ supports a probability measure $\mu$ s.t. 
\begin{equation}
\sum_{n\ne0}\frac{|\widehat{\mu}(n)|^{2}}{|n|^{1-a}}<\infty\label{eq:3}
\end{equation}
then $\dim K\geq a$. 
\item If $\dim K>a$ then $K$ supports a probability measure $\mu$ satisfying
(\ref{eq:3}). 
\end{enumerate}
\end{thm*}
In fact, the first clause may be strengthened to only require that
$K$ supports a (non-trivial) complex-valued measure, or even a distribution
satisfying (\ref{eq:3}). This follows from the following result
\begin{thm*}
[Beurling]If $K$ supports a distribution $S$ satisfying (\ref{eq:3})
then it also supports a probability measure with this property. 
\end{thm*}
\noindent see \cite[Th\'eor\`eme V, \S III]{KS94} or \cite{B49}.
It is worth contrasting this with the result of Piatetski-Shapiro
\cite{P54} that one may find a compact $K$ which supports a distribution
$S$ with $\widehat{S}(n)\to0$ but does not support a measure $\mu$
with $\widehat{\mu}(n)\to0$. We will return to the theorem of Piatetski-Shapiro
in section \ref{sec:Arithmetics-of-compacts}. 

The first result from \cite{KO} which we wish to state is the one-sided
version of the theorem of Frostman \& Beurling:
\begin{thm}
If $K$ supports a distribution $S$ such that
\[
\sum_{n<0}\frac{|\widehat{S}(n)|^{2}}{|n|^{1-a}}<\infty
\]
then $\dim(K)\geq a$. \end{thm}
\begin{proof}
[Proof sketch, step 1]Examine first the case that $a=1$. Then $S$
is a distribution with anti-analytic part belonging to $L^{2}(\mathbb{T})$.
Based on some Phragm\'en-Lindl\"of type theorems in the disc due
to Dahlberg \cite{D77} and Berman \cite{B92}, one can prove that
if $\dim K<1$ then $S=0$. 

\emph{Step 2}. We now reduce the case of general $a$ to the case
of $a=1$. We use Salem's result which gives a quantitative estimate
of the Menshov theorem above. Namely, given $d$, $0<d<1$, there
is a probability measure $\nu$ supported by a compact set $E$ of
Hausdorff dimension $<d+\epsilon$, and such that $\widehat{\nu}(n)\le C|n|^{-d/2}$
. Notice that, as in the original result of Menshov, not every ``thick''
compact supports such a measure. Some arithmetics of $E$ is involved
(this is in stark contrast to Frostman's theorem where only the dimension
plays a role). Take $d=1-a$ and convolve: $S'=S*\nu$. Then $S'$
has its antianalytic part in $L^{2}$. Thus $\dim(K+E)=1$. The theorem
is prove using the inequality $\dim(S+E)\le\dim S+\dim E$.

A difficulty is that it is not true in general that $\dim(A+B)\le\dim A+\dim B$.
A counterexample may be constructed by making $A$ large in some ``scales''
and small in others, and making $B$ large in the scales where $A$
is small and vice versa. Since the Hausdorff dimension is determined
by the best scales, both $A$ and $B$ would have small Hausdorff
dimension but $A+B$ would be large in all scales, and hence would
have large Hausdorff dimension. It is even possible to achieve $\dim A=\dim B=0$
while $\dim(A+B)=1$ \cite[example 7.8]{F03}. However, if $\dim A$
is understood in the stronger sense of upper Minkowski, or box dimension
(which requires smallness in \emph{all} scales), then the inequality
holds. As all the standard constructions of Salem sets in fact have
the same Hausdorff and Minkowski dimensions, the theorem is finished.
\end{proof}
Let us remark that this proof is quite different from the original
proofs of Frostman and Beurling, which were potential-theoretic in
nature.

\section{``Almost analytic'' singular pseudo-functions.}

There is a delicate difference between symmetric and one-sided situations.
The Frostman-Beurling condition in fact implies that compact $K$
has positive $a$-measure, while the one-sided version does not, at
least for $a=1$. The latter was proved in \cite{KO03}
\begin{thm}
\label{thm:2}There is a distribution $S$ with the properties: 
\begin{enumerate}
\item $\widehat{S}(n)=o(1)$.
\item $m(\supp S)=0$.
\item $\sum_{n<0}|\widehat{S}(n)|^{2}<\infty$.
\end{enumerate}
\end{thm}
One can say: a singular pseudo-function can be ``almost'' analytic,
in the sense that its antianalytic part is a usual $L^{2}$ function
(notice that a singular distribution is never analytic, i.e.\ the
antianalytic part cannot be empty). It might be interesting to compare
this with classical Riemannian uniqueness theory, which teaches us
that
\begin{quote}
uniqueness of the decomposition of $f$ in trigonometric series implies
Fourier formulas for coefficients. 
\end{quote}
To make this precise, recall that a set $K$ for which no non-trivial
series may converge to zero outside $K$ is called a set of uniqueness,
or a $\mathcal{U}$-set. With this notation we have
\begin{thm*}
[du Bois-Reymond, Privalov]If a finite function $f\in L^{1}(\mathbb{T})$
has a decomposition in a series (\ref{eq:1}), which converges on
$\mathbb{T}\setminus K$ for some compact $\mathcal{U}$-set $K$
then the series is the Fourier expansion of $f$.
\end{thm*}
See \cite[theorem IX.6.19]{Z68} or \cite{P23}, and \cite[Chap.\ I, \S 72 and Chap.\ XIV, \S 4]{B64}
for the history of the theorem, and a version for $K$ countable but
not necessarily compact. 

In contrast, take $S$ from theorem \ref{thm:2}. Then 

\[
\sum\widehat{S}(n)e^{int}=0\qquad\forall t\in T\setminus K.
\]
Consider the ``anti-analytic'' part $f:=\sum_{n<0}\widehat{S}(n)e^{int}$.
Then $f$ is an $L^{2}$ function on $\mathbb{T}$, smooth on $T\setminus K$.
It admits an ``analytic decomposition'': 

\[
f(t)=\sum_{n\ge0}c(n)e^{int}\qquad c(n):=-\widehat{S}(n)
\]
which converges pointwisely there. Such a representation is \emph{unique},
however it is \emph{not} the Fourier one.

\section{Critical size of the support of ``almost analytic'' distributions}

One can say a bit more about the support of ``one-sided Frostamn's
distributions''. Let $h(t):=tlog1/t$ , and $\Lambda_{h}$ the corresponding
Hausdorff measure. 
\begin{thm}
\label{thm:3}If $S$ is a (non-trivial) distribution such that the
anti-analytic part belongs to $L^{2}(\mathbb{T})$, then $\Lambda_{h}(K)>0$. 
\end{thm}
The result is perfectly sharp. 
\begin{thm}
\label{thm:4}There exist a (non-trivial) pseudo-function $S$ such
that 

\[
\sum_{n<0}|\widehat{S}(n)|^{2}<\infty\qquad\Lambda_{h}(K)<\infty.
\]
\end{thm}
\begin{proof}
[Proof sketch]Let $K$ be a Cantor set on $\mathbb{T}$ of the exact
size, and let $\mu$ be the natural probability measure on $K$. Denote
by $\mu$ also the harmonic extension of $\mu$ into the disc, and
let $\widetilde{\mu}$ be the conjugate harmonic function. Set 
\[
F(z):=e^{\mu+i\widetilde{\mu}}.
\]
In other words, $1/F=e^{-\mu-i\widetilde{\mu}}$ is an inner function
(with the Blaschke term equal to $1$). On the one hand, since $\mu$
is singular, $F$ is unbounded in the disk. On the other hand, its
boundary value $f$ is an $L^{2}$ function --- in fact the boundary
values are exactly $e^{i\widetilde{\mu}}$ so are unimodular. The
constrast between these two properties is what we will use.

First, one can prove that Taylor coefficients $c(n)$ of $F$ at 0
are of polynomial growth. So they correspond to an ``analytic distribution''
on $\mathbb{T}$: 
\[
G:=\sum_{n\geq0}c(n)e^{int}
\]
On the other hand, the boundary value $f$ is an $L^{2}$ function
\[
f=\sum_{n=-\infty}^{\infty}\widehat{f}(n)e^{int}.
\]
Consider the distribution $S$ which is the difference $G-f$. It
is supported by $K$. So it is left to get that $S$ is a ``pseudo-function''. 

How can one get that $\widehat{G}(n)\to0$? One approach is to take
the compact $K$ very very large, and we took this path on our very
first construction of an almost analytic singular pseudo-function
\cite{KO03}. But applying this with a compact $K$ for which $\Lambda_{h}(K)<\infty$
only gives that $\widehat{G}(n)$ converge to $0$ \emph{in average}. 

We solve this problem by letting the components of the Cantor set
$K$ on each step of its construction to ``swim'' a bit around their
canonical positions, randomly. We perform this process heirarchically,
moving each half of the set independently, then moving each quarter
independently with respect to its position in the middle of its half,
and so on. In other words, the position of an interval of order $n$
is the sum of $n$ independent terms, one from each of its ancestors
in the hierarchy defining the Cantor set. 

It is well-known that a random perturbation ``smoothes'' the spectrum.
For example, most of the constructions of Salem sets already mentioned
are random, utilizing this effect, as once the spectrum has been smoothed
the decay of the Fourier coefficients is the best one has given its
$l^{2}$ behaviour, which is known from Frostman's theorem.

The proof has a number of features not shared by previous random constructions,
though. Had we been interested in the Fourier coefficients of $\mu$
itself, then we could have used the locality of the perturbation structure
to write $\widehat{\mu}(n)$ as a sum of \emph{independent} terms
and apply, say, Bernstein's inequality. But we are interested in $e^{\mu+i\widetilde{\mu}}$
and $\widetilde{\mu}$, the complex conjugate of $\mu$ does not have
a local structure. Moving a small piece of $K$ even a little will
change $\widetilde{\mu}$ throughout. 

Another difficulty is that, no matter what probabilistic model one
were to take, it is always true that 
\[
\sum\mathbb{E}|\widehat{G}(n)|^{2}=\infty.
\]
This is because, has this sum been finite we would have that $G$
is in $l^{2}$ almost surely, which is impossible. Hence we are forced
to work with $4^{\textrm{th}}$ moments. Thus we show that $\sum\mathbb{E}|\widehat{G}(n)|^{4}<\infty$,
and conclude that $\widehat{G}(n)\to0$, almost surely, proving the
theorem. The interested reader can see the details of the calculation
of $\mathbb{E}|\widehat{G}(n)|^{4}$ in \cite{KO}, or in \cite{KO06}
where a similar random construction was used.
\end{proof}
Notice that theorems \ref{thm:3} and \ref{thm:4} give a sharp estimate
for the size of the exceptional set of a non-classic analytic expansion. 

A couple of words about the smoothness. It was proved in \cite{KO06}
that the non-analytic half of $S$ in Theorem \ref{thm:2} can be
infinitely smooth, that is $\widehat{S}(-n)=o(1/n^{k})$ for every
$k$. In fact, we have a precise description of the exact maximal
smoothness achievable, or in other words, a quasi-analyticity-like
result. Surprisingly, while in classical quasi-analyticity results
the ``critical smoothness'' is approximately $\widehat{S}(n)\asymp e^{-n/\log n}$,
i.e.\ quite close to analytical smoothness, in our case the maximal
smoothness allowable is approximately $\widehat{S}(-n)\asymp e^{-\log n\log\log n}$
i.e.\ only slightly above $C^{\infty}$ smoothness.

It seems that one has to pay for smoothness by the size of support,
which probably must increase. 

\begin{qst*}How does the ``critical size'' of the support of $K$
depend on the order of smootness?\end{qst*}

\section{Non-symmetry for measures}

Theorem \ref{thm:4} reveals a \emph{non-symmetry} phenomenon for
\emph{distributions}. A singular complex measure can not be ``almost
analytic'' according to the theorem of the brothers Riesz. Measures
in general have more symmetry. Example (Rajchman \cite{R29} or \cite[\S 1.4]{KL87}):
If the Fourier coefficients of a complex measure on the circle are
$o(1)$ at positive infinity then the same is true at negative infinity.
Another result of this sort (Hru\v s\v cev and Peller \cite{HP86}, Koosis and Pichorides, see \cite{K83}):
\begin{thm*}
If $\sum_{n>0}|\widehat{\mu}(n)|^{2}/|n|<\infty$, then the same is
true for sum over $n<0$.
\end{thm*}
However, certain non-symmetry may happen even for singular measures. 
\begin{thm}
Given $d>0$, $p>2/d$ there is a measure $\mu$ supported on a compact
$K$ such that
\begin{enumerate}
\item $\dim K=d$; 
\item $\widehat{\mu}\in l_{p}(\mathbb{Z}^{-})$ but not in $l^{p}(\mathbb{Z}^{+})$
\end{enumerate}
\end{thm}
See \cite[theorem 3.1]{KO}.\begin{rmrks*}

1. The restriction on $p$ is sharp, due to theorem \ref{thm:1}.

2. Not every $K$ may support such a measure, $K$ has to be a Salem
set (in the same sense as above). \end{rmrks*}

\begin{qst*}Let $K$ supports $S$ with $\widehat{S}(n)\to0$ as
$n\to\infty$. Does it support an $S'$ with the two-sides condition
$|n|\to\infty$? (in other words, can a $\mathcal{U}$-set support
a distribution with $\widehat{S}(n)\to0$ as $n$ tends to positive
$\infty$?)\end{qst*}

\section{\label{sec:Arithmetics-of-compacts}``Arithmetics'' of compacts}

It was already mentioned that not only the size of the support but
rather its arithmetic nature may play a crucial role in the behaviour
of measures or distributions. The most remarkable illustration of
this phenomenon is related to the following problem. Let $K_{\theta}$
be the symmetric Cantor set with dissection ratio $\theta$, i.e.\ the
set one gets by starting with a single interval and then repeatedly
replacing each interval of length $a$ by two intervals of length
$\theta a$. When may $K_{\theta}$ satisfy the ``Menshov property'',
that is may support a measure or distribution with Fourier transform
vanishing at infinity? It was Nina Bari who discovered (in the case
when $\theta$ is rational) that the answer has nothing to do with
the size of the compact and depends on arithmetics of $\theta$. This
result is not so often mentioned in western literature. The much deeper
Salem-Zygmund theorem extends this phenomenon to irrational $\theta$.
The result is that $K_{\theta}$ supports a null-measure if and only
if the dissection ratio is not the inverse of a Pisot number, i.e.\ an
algebraic integer $>1$ all whose algebraic conjugates are $<1$ in
absolute value.

Relations between uniqueness theory and number theory are by now a
well-developed theory, see the book \cite{M72}. Let us also recall
the deterministic constructions of Salem sets \cite{Kauf81}, the
role of the so-called Kronecker sets \cite[chaptrer 7]{K70} and results
on the set of non-normal numbers \cite{P54,KS64,L86}. 

We'll finish this survey by a recent result in which the arithmetics
of a compact also plays a crucial role. It relates to the so-called
Wiener problem on cyclic vectors. 
\begin{defn*}
Let $1\leq p\leq\infty$. A function $x\in l^{p}(\mathbb{Z})$ is
called a cyclic vector if its translates spans the whole space. 
\end{defn*}
Wiener characterized cyclic vectors for $p=1$ and $p=2$:
\begin{thm*}
\hfill
\begin{enumerate}
\item $x=\{x_{k}\}$ is cyclic in $l^{1}$ iff $\widehat{x}(t):=\sum x_{k}e^{ikt}$
has no zeros; 
\item x is cyclic in $l^{2}$ iff $\widehat{x}(t)\neq0$ a.e. 
\end{enumerate}
\end{thm*}
Wiener conjectured \cite[page 93]{W32} that for every $p$, or at
least for $p<2$, cyclic vectors can be characterized by a certain
``negligibility'' condition of the zero set of $\widehat{x}$. It
turned out this is not the case:
\begin{thm}
[\cite{LO11}]Let $1<p<2$. Then there are two vectors $x$ and $y\in l^{1}$
such that
\begin{enumerate}
\item The zero sets of $\widehat{x}$ and $\widehat{y}$ are the same; 
\item $x$ is cyclic in $l^{p}$, while $y$ is not. 
\end{enumerate}
\end{thm}
The approach to the proof is based on certain development of ideas
of Piatetskii-Shapiro \cite{P54} who aparently was the first who
inserted functional-analytic and probabilisic ideas into Riemannian
uniqueness theory.

\end{document}